\documentclass[a4paper, 11pt]{amsart}

\usepackage{amsmath, amssymb, amsthm}
\usepackage[english]{babel}
\usepackage[utf8]{inputenc}
\usepackage{enumerate}
\usepackage{xcolor}
\newcommand{\ee}{\mathbf{e}}

\def \be {  \varpi}

\newcommand{\fer}[1]{(\ref{#1})}

\def\MF{ {\mathcal {F}}}
\def\MG{ {\mathcal {G}}}
\def\MI{ {\mathcal {I}}}

\newcommand{\R}{\mathbb R}
\def\be#1\ee{\begin{equation}#1\end{equation}}

\newcommand{\bq}{\begin{equation}}
\newcommand{\eq}{\end{equation}}

\newtheorem{thm}{Theorem}

\theoremstyle{remark}
\newtheorem{rem}[thm]{Remark}
\theoremstyle{definition}

 \numberwithin{equation}{section}

\usepackage{verbatim}

\newenvironment{equations}{\equation\aligned}{\endaligned\endequation}

\title[Fokker--Planck equations and inequalities for heavy tailed densities]{Fokker--Planck equations and one--dimensional functional inequalities for heavy tailed densities }
\author{GIULIA FURIOLI}
\thanks{DIGIP, University of Bergamo, viale Marconi 5, 24044 Dalmine, Italy\\
giulia.furioli@unibg.it}
\author{ADA PULVIRENTI}
\thanks{Department of Mathematics, University of Pavia, 
via Ferrata 5,
Pavia, 27100 Italy\,\,\,\\
ada.pulvirenti@unipv.it}
\author{ELIDE TERRANEO}
\thanks{Department of Mathematics,
University of Milan, via Saldini 50, 20133 Milano, Italy \\
elide.terraneo@unimi.it}
\author{GIUSEPPE TOSCANI}
\thanks{Department of Mathematics, University of Pavia, and IMATI of CNR
via Ferrata 5,
Pavia, 27100 Italy,\\
giuseppe.toscani@unipv.it}
\begin{document}
\maketitle
%
%
\begin{center}\small
\parbox{0.85\textwidth}{

\textbf{Abstract.} 
We study one-dimensional functional inequalities of the type of  Poincar\'e, logarithmic Sobolev and Wirtinger, with weight, for probability densities with polynomial tails.   As main examples, we obtain sharp inequalities satisfied by inverse Gamma densities,  taking values on $\R_+$, and  Cauchy-type  densities, taking values on $\R$. In this last case, we improve the result obtained by Bobkov and Ledoux in 2009 by introducing a better weight function in the logarithmic Sobolev inequality. The results are obtained by resorting to Fokker--Planck type equations which possess these densities as steady states. }
\medskip


\medskip


\end{center}

\section{Introduction}\label{intro}

This paper deals with one-dimensional functional inequalities of the type of  Poincar\'e, logarithmic Sobolev and Wirtinger, with weight, for probability densities with polynomial tails. These inequalities are strongly related with the problem of convergence to equilibrium for Fokker--Planck type equations, where the main examples are furnished by the inverse Gamma and Cauchy-type probability densities.  

Let $X$ be a random variable distributed with probability density $f(x)$, where  $x\in \MI \subseteq \R$. The random variable $X$ is said to satisfy a weighted Poincar\'e-type inequality with weight function $w(x)$ (where $w$ is a fixed nonnegative, Borel measurable function), if for any bounded smooth function $\phi$ on $\MI$ 
 \be\label{Poi}
 Var\left[\phi(X)\right] \le E\left\{w(X)[\phi'(X)]^2\right\}.
 \ee
 As usual, for a given random variable $Y$, $E(Y)$ denotes its expectation value, and
 \[
 Var\left[\phi(X)\right] = \int_{\MI} \phi^2(x)\, f(x) \, dx - \left(\int_{\MI} \phi(x)\, f(x) \, dx \right)^2
 \]
is the variance of $\phi$ with respect to $f$.  
 Likewise, $X$ is said to satisfy a weighted logarithmic Sobolev inequality with weight function $w(x)\ge 0$ if, for any bounded smooth function $\phi$ on $\MI$ 
 \be\label{PoiS}
 Ent\left[\phi^2(X)\right] \le E\left\{w(X)[\phi'(X)]^2\right\}.
 \ee
 Here
 \[
 Ent \left[\phi^2(X)\right] = \int_{\MI} \phi^2(x)\log \phi^2(x)\, f(x) \, dx - \left(\int_{\MI} \phi^2(x)\, f(x) \, dx\right ) \, \log \left(\int_{\MI} \phi^2(x)\, f(x) \, dx \right )
 \]
 denotes the entropy of $\phi^2$ with respect to $f$. 
 Last, the random variable $X$ is said to satisfy a weighted Wirtinger-type inequality with weight function $K(x)$ (where $K$ is a fixed nonnegative, Borel measurable function), if for any bounded smooth function $\phi$ on $\MI$, and $p \ge 1$
 \be\label{Wir}
 E\left\{|\phi(X)|^p\right\}] \le E\left\{K(X)^p|\phi'(X)|^p\right\}.
 \ee

 The inequalities are understood in the following sense: if the right-hand side is finite, then  the inequalities hold true.

Abstract weighted Poincar\'e and logarithmic Sobolev inequalities are connected with the problem of large deviations of Lipschitz functions and measure concentration. In reason of that, the question whether a probability measure satisfies such functional inequalities has attracted a lot of attention in recent years  \cite{BCG,BL,BJ,BJM1,BJM2,CGGR,Goz}. 

In the probabilistic literature, inequality \fer{Poi} is also known under the name of weighted Chernoff inequality, in reason of the analogous inequality with weight $w(x) =1$ obtained by Chernoff \cite{Che} for the one-dimensional Gaussian density 
 \be\label{Max}
 g(x) = \frac 1{\sqrt{2\pi}} \exp\left\{-\frac{x^2}2\right\}, \qquad x \in \R.
 \ee
Chernoff-type inequalities with weight were proven, few years later Chernoff's result, by Klaassen \cite{Kla}, who listed a number of probability densities for which the weight $w(x)$ was explicitly computable.  A different proof of Chernoff-type inequalities with weight, valid for heavy tailed densities, has been recently obtained in \cite{FPTT} by resorting to the representation of these densities as equilibria of Fokker--Planck type equations with variable coefficient of diffusion and linear drift. 

The analysis of \cite{FPTT} was deeply motivated by the study of Fokker--Planck type equations appearing in the modeling of social and economic phenomena,  a challenging research activity in the communities of both physicists and applied mathematicians, who classified the fields of research with the names of socio-physics and, respectively, econophysics \cite{FPTT, NPT, PT13}. 

One of the typical features of these phenomena is related to the tails of the underlying steady distribution, which are often characterized by polynomial decay at infinity \cite{New}.
The classical example is furnished by  the study of the distribution of wealth among trading agents, that leads to a Fokker--Planck type equation with variable coefficients of diffusion and linear drift \cite{BM, CoPaTo05}. This equation, which describes the time-evolution of the density $f(x,t)$ of a system of agents with personal wealth $x\ge0$ at time $t \ge 0$ reads
  \be\label{FPw}
 \frac{\partial f(x, t)}{\partial t} = \frac \sigma{2}\frac{\partial^2 }
 {\partial x^2}\left( x^2 f(x,t)\right) + \lambda \frac{\partial }{\partial x}\left(
 (x-1) f(x,t)\right).
 \ee
In \fer{FPw},  $\sigma$ and $\lambda$ denote  positive constants related to essential properties of the trade rules of the agents.  
By fixing the mass density equal to unity, the unique steady state  of equation \fer{FPw} is the inverse Gamma density 
 \be\label{equi2}
f_\infty(x) =\frac{\mu^{1+\mu}}{\Gamma(1+\mu)}\frac{\exp\left(-\frac{\mu}{x}\right)}{x^{2+\mu}},
 \ee
 characterized by the positive constant $\mu$,  given by
  \be\label{mu}
  \mu = 2 \frac{\lambda}\sigma.
 \ee
This stationary distribution, in agreement with the analysis of the Italian economist Vilfredo Pareto \cite{Par}, exhibits a power-law tail for large values of the wealth  variable.

One of the physically relevant questions related to the Fokker--Planck equation \fer{FPw} is the know\-ledge of the exact rate of relaxation to equilibrium of its solution, which,  as it happens for the classical Fokker--Planck equation \cite{To99}, is expected to be exponential in time.  This rate of relaxation would in fact justify that in real economies the  wealth distribution profile is always well fitted, for large values of the wealth variable, by  the inverse Gamma \cite{PT13}.

The study of relaxation to equilibrium for the solution to equation \fer{FPw} treasured the successful methodology  used for the classical Fokker--Planck equation, corresponding to constant coefficient of diffusion and linear drift.  Hence,  the relaxation of the solution of \fer{FPw} towards equilibrium was usually investigated by looking at the time evolution of its Shannon entropy relative to the equilibrium density \cite{FPTT}.  We recall that, given two probability densities $f(x)$ and $g(x)$, with $x \in \mathcal I \subseteq \R$, the Shannon entropy of $f$ relative to  $g$ is defined by
 \be\label{rel-H}
H(f|g) = \int_{\mathcal I} f(x) \log \frac{f(x)}{g(x)}\, dx.
 \ee
This argument has its roots in classical statistical physics, and, similarly to the classical kinetic theory of rarefied gases, identifies the steady state solution of the Fokker--Planck equation \fer{FPw} as the target density to be reached monotonically in time in relative entropy \cite{FPTT}. 
The proof of the exponential convergence of the solution to the classical Fokker--Planck equation towards the Maxwellian (Gaussian) density \fer{Max} follows by applying the logarithmic Sobolev inequality \cite{To97, To99}, which establishes a sharp bound of the relative entropy  in terms of the entropy production. This introduces 
a deep link between differential inequalities of Sobolev type and Fokker--Planck  equations. While the standard logarithmic Sobolev inequality allows us to prove exponential convergence of the solution towards the Maxwellian equilibrium density \fer{Max} in relative entropy, at the same time the evolution of the relative entropy of the solution density of the Fokker--Planck equation can be used to obtain a dynamical proof of the logarithmic Sobolev inequality  \cite{To97,To99}. 
 This idea has been subsequently extended, to obtain sharp differential inequalities,  to  Fokker--Planck type equations with constant diffusion term and general drift by Otto and Villani \cite{OV}. 

As a matter of fact, however, even if various weaker results are available \cite{TT},  a proof of exponential convergence in relative entropy of the solution to the Fokker--Planck equation \fer{FPw} towards the inverse Gamma \fer{equi2} is at present not available. In a recent paper \cite{FPTT20} a possible motivation of this unpleasant difference in convergence between the classical and the wealth Fokker--Planck equations has been identified in the choice of a Maxwellian (constant) interaction kernel,  made in \cite{CoPaTo05},  in the kinetic equation leading to \fer{FPw}. Without going to detail regardind this discussion about the modeling assumptions, that the interested reader can find in \cite{FPTT20},  the introduction of a variable collision kernel  led to build a new Fokker--Planck equation with variable coefficient of diffusion and variable drift, still describing the time-evolution of the density $f(x,t)$ of a system of agents with personal wealth $x\ge0$ at time $t \ge 0$.  
This new Fokker--Planck equations reads
\begin{equation}\label{FPz}
 \frac{\partial f(x,t)}{\partial t} =   \frac\sigma 2 \frac{\partial^2 }{\partial x^2}
 \left(x^{2+\delta}f(x,t)\right )+ 
\lambda\, \frac{\partial}{\partial x}\left( x^\delta(x-1 ) f(x,t)\right).
 \end{equation}
 In \fer{FPz}  $\delta$ is a positive constant, with $0<\delta\leq1$. Equation \fer{FPz} has a unique equilibrium density of unit mass, still given by an inverse Gamma function
 \be\label{new-equ}
 f_\infty^\delta(x) = \frac{\mu^{1+\delta + \mu}}{\Gamma(1+\delta + \mu)}\frac{\exp\left(-\frac{\mu}{x}\right)}{x^{2+\delta +\mu}}.
 \ee
In \fer{new-equ} $\mu$ is the positive constant defined in \fer{mu}. Hence, the presence of the constant $\delta$ is such that the Pareto index in the equilibrium density of the target Fokker--Planck equation is increased by the amount $\delta>0$. Note that this class of Fokker--Planck type equations contains \fer{FPw}, which is obtained in the limit $\delta \to 0$.

As proven in \cite{FPTT20},  and in contrast to equation \fer{FPw}, the solution to the Fokker--Planck equation \fer{FPz}   has been shown to converge exponentially, with explicit rate, towards the equilibrium  density \fer{new-equ}. For this reason, equation \fer{FPz} has been proposed as a better model for the description of the process of relaxation of the wealth distribution density in a multi-agent society \cite{FPTT20}.

A critical comparison of the two Fokker--Planck equations \fer{FPw} and \fer{FPz} allows us to come to some interesting conclusions, that will be at the basis of the results of this paper. In view of the previous discussion, the inverse Gamma density can appear as the steady state of different Fokker--Planck equations, which can share several properties in relation with convergence to equilibrium. Indeed, by looking at the computations in \cite{FPTT20}, the evolution of the  Shannon entropy of the solution to the Fokker--Planck equation \fer{FPz} relative to the equilibrium solution depends on the parameter $\delta$, that appears in the entropy production term. This suggests that in order to get sharp differential inequalities for a certain probability density with heavy tails, one has to look for the most general class of Fokker--Planck type equations which possess this probability density as steady state, aiming in finding the optimal one. 

It is interesting to remark that this strategy is not restricted to differential inequalities of Sobolev type, but it can be fruitfully applied also to Chernoff (Poincar\'e) type inequalities, thus generalizing the result obtained in \cite{FPTT}. 

In addition to the class of inverse Gamma functions, this new method will be applied to obtain weighted inequalities for the class of Cauchy-type densities $f_\beta(x)$, $\beta >1/2$,  with $x \in \R$
\be\label{cau}
f_\beta(x)= \frac {C_\beta}{(1+x^2)^\beta}.
\ee
Weighted inequalities for Cauchy-type densities have been studied by Bobkov and Ledoux in  dimension $d \ge 1$ in \cite{BL} with a different technique. In one dimension of the space variable, by resorting to this relationship with Fokker--Planck type equations, we will  improve in some cases the weight function obtained in \cite{BL} in the logarithmic Sobolev inequality. Also, we will show that the optimal results in weighted Poincar\'e inequality for Cauchy-type densities obtained in \cite{BJM2} can follow by resorting to this idea. 

The content of the paper is as follows. In Section \ref{sec:chernoff} we will prove an extension of the classical Chernoff inequality obtained in \cite{FPTT}, and we apply the result to Cauchy-type and inverse Gamma densities. Likewise, Section \ref{sec:sobolev} and will contain the results about weighted logarithmic Sobolev inequalities for the same classes of probability densities. Both Sections \ref{sec:chernoff} and \ref{sec:sobolev} will take advantage of the representation of Cauchy-type and inverse Gamma densities as steady solutions to Fokker--Planck type equations of type \fer{FP-gen}. 
Finally, Section \ref{sec:wirtinger} will contain an improvement of an old result by Elcrat and MacLean \cite{EM} about weighted Wirtinger inequalities on unbounded domains, with application to the Cauchy-type and inverse Gamma functions,  densities for which explicit weight functions are obtained. In this case,  the underlying probability density is characterized as the equilibrium density of a Fokker--Planck  with a positive bounded coefficient of diffusion and an elementary coefficient of drift.

To end this introduction, it is important to outline that the strategy used in this paper can be fruitfully used to obtain differential inequalities with weight for other densities.


\section{Chernoff--type inequalities for heavy tailed densities}\label{sec:chernoff}

 The aim of this Section is to prove by means of their relationship with Fokker--Planck type equations that the class of densities \fer{new-equ} and \fer{cau} satisfy some sharp weighted inequalities of  Chernoff  type. In the rest of this Section, we refer to  a class  of Fokker--Planck type equations with variable coefficients of diffusion and drift in the form
\begin{equation}\label{FP-gen}
 \frac{\partial f(x,t)}{\partial t} =   \frac{\partial^2 }{\partial x^2}
 \left(P(x)f(x,t)\right )+ 
\, \frac{\partial}{\partial x}\left( Q(x) f(x,t)\right),
 \end{equation}
where $x \in \MI =(i_-,i_+) \subseteq \R$.  In equation \fer{FP-gen} the coefficients of the diffusion $P(x)$  and the drift $Q(x)$ are smooth  functions, and $P(x) \ge 0$. We suppose moreover that $P(x)$ and $Q(x)$ are such that, for any $a\in \MI$, the steady state 
 \be\label{eq}
 f_\infty(x) = \frac C{P(x)} \exp \left\{  - \int_a^x\, \frac{Q(y)}{P(y)}\, dy\right\}.
 \ee
 is a probability density supported in $\MI$ for a given value of the constant $C$.
 Note that the steady state \fer{eq} satisfies the first order differential equation
\begin{equation}\label{ste1}
   \frac{\partial }{\partial x}
 \left(P(x)f_\infty(x)\right )+ 
\,  Q(x) f_\infty(x) = 0.
 \end{equation}
The case 
\be\label{facile}
Q(x) = x-M, \quad M \in (i_-,i_+)
\ee
has been considered and studied in \cite{FPTT}, resorting to a clear  proof that is closely related to the Fokker--Planck description of the equilibria, as given by \fer{ste1}. In \cite{FPTT} it was proven that, 
if $X$ is a random variable distributed with density $f_\infty(x)$, $x \in \MI \subseteq \R$, and $f_\infty$ satisfies the differential equality
\be\label{staz-2}
\frac{\partial }{\partial x}\left(P(x)  f_\infty(x) \right) + (x -M)\,f_\infty(x) = 0, \quad x\in \MI, 
 \ee
then for any smooth function $\phi$  defined  on $\MI$ such that $\phi(X)$ has finite variance  
\be\label{ch-gen}
 Var[\phi(X)] \le E\left\{P(X)[\phi'(X)]^2\right\}
\ee
 with equality if and only if $\phi(X)$ is linear in $X$.  Note that when $Q(x)$ is linear in $x$, the weight function $w$  coincides with the variable coefficient of diffusion $P(x)$. In what follows, we extend the result of \cite{FPTT} to cover a larger class of functions $Q(x)$.

\subsection{An extension of Chernoff-type inequality} 

\begin{thm}[Chernoff with weight]\label{C}
Let $X$ be a random variable distributed with density $f_\infty(x)$, $x \in \MI =(i_-, i_+)\subseteq \R$. Let us suppose moreover that $f_\infty$ satisfies \fer{ste1}, where $Q(x)$ is a smooth function strictly increasing on $\MI$ such that \be\label{bordi}
 \lim_{x\to i_-} Q(x)<0,\quad \lim_{x\to i_+} Q(x) >0.
 \ee
 Let $w(x)$ be defined by  
 \be\label{mag}
 w(x) = \frac {P(x)}{Q'(x)},\quad x\in \MI.
 \ee
 Then,  for any  smooth function $\phi$  on $\MI$, and $\phi(X)$ with finite variance  
\[
 Var[\phi(X)] \le E\left\{w(X)[\phi'(X)]^2\right\},
\]
that is 
 \be\label{Ch-gen}
 \int_\MI \phi^2(x) f_\infty(x) \, d x -\left(  \int_\MI \phi(x) f_\infty(x) \, d x\right )^2 \leq   \int_\MI w(x)  \left(\phi'(x)\right )^2 f_\infty(x) \, d x.
 \ee
\end{thm}

\begin{proof}

Let $\phi$ be a  smooth function  on $\MI$ such that  $Var(\phi(X))$ is bounded.
Since $\int_\MI f_\infty(x) dx =1$, for any given constants $A\in \R$ it holds
\[
Var(\phi(X)) = \int_\MI \phi^2(x) f_\infty(x) \, d x -\left(  \int_\MI \phi(x) f_\infty(x) \, d x\right )^2 \leq \int_\MI(\phi(x)-A)^2 f_\infty(x) dx.
\]
Now, since $Q(x)$ is strictly increasing on $\MI$, and satisfies conditions \eqref{bordi},   there is a point $x_0\in \MI$ where $Q(x_0)=0$. Let us consider the change of variable $x \to Q(x)$, which is invertible for $x \in \MI$ due to the assumptions on $Q$ and let us define
\be\label{psi}
\phi(x)= \psi(Q(x)). 
\ee
If we set  $A= \phi(x_0)= \psi(0)$,  we get
\[
\begin{aligned}
\int_\MI(\phi(x)-\psi(0))^2 f_\infty(x) )dx &= \int_\MI( \psi(Q(x))-\psi(0))^2 f_\infty(x) dx \\
=\int_\MI  \left(\int_0^{Q(x)} \psi'(s) ds\right )^2 f_\infty(x) dx
&= \int_\MI  \left(\int_0^{1} \psi'(t\, Q(x)) Q(x) dt\right )^2 f_\infty(x) dx.
\end{aligned}
\]
Now by Jensen's inequality
\[
\begin{aligned}
 \int_\MI  f_\infty(x)  \left(\int_0^{1} \psi'(t\,Q(x))Q(x) dt\right )^2 dx & \leq  \int_\MI  f_\infty(x) \left( \int_0^{1}  \big( \psi'(t\, Q(x)) Q(x)\big)^2 dt \right )dx\\
&=  \int_\MI  f_\infty(x) Q^2(x)  \left(  \int_0^{1} \left(\psi'(t\, Q(x))\right)^2 dt\right ) dx.
\end{aligned}
\]
Using that $f_\infty$ satisfies \fer{ste1}  we get
\begin{equations} \label{bello}
&\int_\MI  f_\infty(x) Q^2(x)  \left(  \int_0^{1} \left(\psi'(t\, Q(x))\right)^2 dt\right ) dx \\
& = -  \int_\MI \partial_x \big(P (x)f_\infty(x)\big) Q(x) \left(  \int_0^{1} \left(\psi'(t\, Q(x))\right)^2 dt\right) dx 
  \\
  &= \left[-P(x)f_\infty(x) Q(x) \left(  \int_0^{1} \left(\psi'(t\,Q(x))\right)^2 dt  \right ) \right]_{i_-}^{i_+} \\ & +  \int_\MI P(x)f_\infty(x) \partial_x\left(Q(x)  \int_0^{1} \left(\psi'(t\,Q(x))\right)^2 dt\right ) dx.
  \end{equations}
The boundary term in \fer{bello}, due to assumption \eqref{bordi} is non positive. In view of the identity 
 \[
\partial_y ( y\beta (t y))= \partial_t(t\beta (t y)), 
 \] 
 valid for any function $\beta(\cdot)$, and variables $y$ and $t$ we get
  \begin{equations}\label{bello2}
   &\int_\MI P(x)f_\infty(x) \partial_x\left(Q(x)  \int_0^{1} \left(\psi'(t\,Q(x))\right)^2 dt\right ) dx\\
  & =  \int_\MI P(x)f_\infty(x) Q'(x) \left.\partial_y \left(y \int_0^{1} \left(\psi'(t\, y)\right)^2  dt\right )\right|_{y =Q(x)} dx\\
  &=  \int_\MI P(x)f_\infty(x) Q'(x)\left.  \int_0^{1}   \partial_y \big( y\psi'(t\, y)^2 \big)\right|_{y =Q(x)}dt   dx\\
  &=  \int_\MI P(x)f_\infty(x) Q'(x)  \int_0^{1}   \partial_t \left(t \psi'(t\, Q(x))^2\right) dt   dx\\
  &=  \int_\MI P(x)f_\infty(x) Q'(x) \psi'(Q(x))^2 dx.
  \end{equations}
  Now, differentiation of \eqref{psi} gives
  \[
  \phi'(x) = \psi'(Q(x)) Q'(x),
  \]
  and 
  \be\label{psi-der}
   \psi'(Q(x)) = \frac {\phi'(x)}{Q'(x)}.
   \ee
   Replacing equality \fer{psi-der} into the last integral in \fer{bello2}, and using the relation \eqref{mag} it follows that
  \[
   \int_\MI P(x)f_\infty(x) Q'(x) \psi'(Q(x))^2 dx   =  \int_\MI w(x)
   \left( \phi'(x)\right)^2  f_\infty(x)dx.
   \]
   Finally we have
   \[
   Var(\phi(X)) \leq   \int_\MI w(x)
   \left( \phi'(x)\right)^2  f_\infty(x)dx,
   \]
   and the proof is completed.
\end{proof}
\begin{rem} Even if the main applications of Theorem \ref{C} refer to probability densities with heavy tails, it is interesting to remark that the case $P(x) = 1$ leads to a functional inequality related to weighted Poincar\'e inequalities, known as the Brascamp-Lieb inequality \cite{BLieb}. If $P(x) =1$, the steady state $f_\infty$ takes the form
\be\label{eq-lieb}
 f_\infty(x) =  C \exp \left\{  - \int_a^x\, Q(y)\, dy\right\},
 \ee
and the condition $Q'(x) >0$ of Theorem \ref{C} corresponds to assume that the potential 
 \[
 V(x) = \int_a^x\, Q(y)\, dy
 \]
is strictly convex. In this case inequality \fer{Ch-gen} can be written in terms of the strictly convex potential $V(x)$ to give
 \be\label{BLieb}
  Var[\phi(X)] \le  \int_\MI \frac 1{V''(x)}  \left(\phi'(x)\right )^2 f_\infty(x) \, d x.
 \ee
Inequality \fer{BLieb} is exactly the Brascamp-Lieb inequality in dimension one \cite{BLieb}.
\end{rem}

Theorem \ref{C} allows us to prove Chernoff-type inequalities with weight for various families  of  probability densities on the line with heavy tails.  We present below two examples, which refer to the family of the Cauchy--type densities \fer{cau} in the range $\beta >1/2$,   and to the family of inverse Gamma densities \fer{equi2}. To maintain the analogy with the Cauchy-type densities, we will write the inverse Gamma densities in the form
 \be\label{inv-gamma}
 h_{\beta,m}(x)= \frac {C_{\beta, m}}{x^{2\beta}} e^{-\frac mx},\quad x\in\R_+.
 \ee
where $\beta >1/2$ and $m>0$. The constant $C_{\beta,m}$  is explicit, and it is such that the functions $f_{\beta, m}$ have unit mass. 

It is interesting to remark that, as far as the Cauchy-type densities are concerned,  the same  inequalities with weight  have been recently obtained in \cite{BJM2}, by resorting to  the spectral gap of a convenient Markovian diffusion operator, and then using a recent result \cite{BJ}, which allows the authors to estimate precisely this spectral gap.  While for $\beta >3/2$ the present proof, which in our setting corresponds to choosing $Q(x) = x$, is similar to that of \cite{BJM2}, the method of proof in the case $1/2 < \beta < 3/2$ is new, and makes a substantial use of Theorem \ref{C}.

\subsection{ Chernoff with weight for Cauchy--type densities}

\begin{thm} [Chernoff for Cauchy--type densities]\label{Ch-Cauchy}

Let $X$ be a random variable distributed with the Cauchy-type density \fer{cau}, with  $\beta >1/2$. For any smooth function $\phi(x)$, with $x\in\R$ such that  $\phi(X)$ has finite variance, one has the bounds
\be
\label{chernoff-gen}
 Var[\phi(X)] \le \frac 1{\rho(\beta)} E\left\{(1+X^2)[\phi'(X)]^2\right\},
\ee
 where
\be
\rho(\beta)= 
\begin{cases}
\left(\beta-\frac 12\right )^2 & \frac12 <\beta \leq \frac 32\\
2(\beta-1) & \beta > \frac 32.
\end{cases}
\ee
\end{thm}

\begin{proof}
We start by remarking that the Cauchy-type density $f_\beta$ defined in \fer{cau} can be characterized as the stationary state of a whole family of Fokker--Planck type equations, which depend on two positive parameters $\alpha$ and $\lambda$ related to satisfy the constraint
\be\label{a-la}
\alpha(1+\lambda)=\beta.
\ee
For our purposes, we will assume that  $\alpha \in \left( 1/2, 1\right]$ and let $\lambda >0$. This choice guarantees that  we can obtain from relation \fer{a-la} all values of $\beta >1/2$. It can be easily checked that 
this family of Fokker--Planck type equations is given by
\[
\partial_t f= \partial^2_x \left((1+x^2)^\alpha f\right ) +\lambda \partial_x\left(\frac {2\alpha x}{(1+x^2)^{1-\alpha}} f\right), \quad x\in \R, t>0.
\]
Indeed $f_\beta$ satisfies, for all $x\in\R$, the differential equation
\be\label{sta-b}
\partial_x \left((1+x^2)^\alpha f_\beta\right ) =-\frac {2\alpha \lambda x}{(1+x^2)^{1-\alpha}} f_\beta.
\ee
Equation \fer{sta-b} is of the type \fer{ste1}, with 
 \be\label{pq}
P(x) = (1+x^2)^\alpha, \quad Q(x)=\frac {2\alpha \lambda x}{(1+x^2)^{1-\alpha}}, \quad x\in\R.
\ee
In the allowed range of the constant $\alpha$, the function $Q$ satisfies all the assumptions of Theorem \ref{C}.
Indeed, the function $Q$ is differentiable on $\R$ and for all $\alpha >1/2$
\[
Q'(x)= \frac {2\alpha \lambda (1+x^2(2\alpha -1))}{(1+x^2)^{2-\alpha}} >0, \quad  x\in \R.
\]
Moreover
\[
\lim_{x\to -\infty} Q(x)=-\infty, \quad \lim_{x\to +\infty} Q(x)= +\infty.
\]
So $Q : \R \longrightarrow \R$ is a strictly monotone, smooth transformation.
Moreover, since $\alpha \leq 1$ we have $2\alpha -1 \leq1$,   so that
\be\label{bound-c}
\frac{P(x)}{Q'(x)} = \frac{(1+x^2)^2}{2\alpha \lambda (1+x^2(2\alpha-1))} \leq  \frac 1{2\alpha\lambda(2\alpha-1)} (1+x^2),\quad x\in\R.
\ee
Therefore, from inequality \fer{Ch-gen} of Theorem \ref {C} we obtain
\[
 \int_\R \phi^2(x) f_\beta(x) \, d x -\left(  \int_\R \phi(x) f_\beta(x) \, d x\right )^2 \leq \frac 1{2\alpha \lambda (2\alpha -1)}  \int_\R (1+x^2)  \left(\phi'(x)\right )^2 f_\beta(x) \, d x.
\]
We can now look for the optimal value of the constant $2\alpha \lambda (2\alpha -1))$ under the constraints  $\alpha \in \left (1/2, 1\right ]$, $\lambda >0$ and  \fer{a-la}.
Since \fer{a-la} implies $\lambda =\beta/ \alpha -1$, the optimal value is obtained by maximizing the function
\[
\rho_\beta(\alpha)= 2(\beta-\alpha)(2\alpha-1).
\]
To this end, since
\[
\rho'_\beta(\alpha)= 2(2\beta-4\alpha+1)
\]
we obtain
\[
\rho_\beta'(\alpha) \geq 0 \iff \alpha \leq \frac \beta 2+\frac 14.
\]
If $\frac 12 < \frac \beta 2 +\frac 14 \leq 1$, then $\alpha_{\max}= \frac \beta 2 +\frac 14$,  while if $ \frac \beta 2 +\frac 14>1$, then $\alpha_{\max}=1$.
Denoting by $\rho(\beta)= \max\{\rho_\beta(\alpha), \frac 12 <\alpha \leq 1\}$, we then find 
\be\label{optimal}
\rho(\beta)= 
\begin{cases}
\left(\beta-\frac 12\right)^2& \frac 12 <\beta \leq \frac 32\\
2(\beta-1) &\beta >\frac 32.
\end{cases}
\ee
This completes the proof.
\end{proof}

\begin{rem}  In the case $\beta > 3/2$ the optimal constant $\rho(\beta)$  is obtained by choosing $\alpha_{\max} =1$. In this case $Q(x) = 2\lambda x$ is therefore linear, and the proof of Chernoff inequality was already obtained in \cite{FPTT}. Moreover, in this range of the parameter $\beta$, the Cauchy-type density has finite variance and this implies that the function $\phi$ in Theorem \ref{Ch-Cauchy} can be chosen to be linear in $x$. Since we proved in \cite{FPTT} that Chernoff inequality with weight further guarantees that there is equality in \fer{chernoff-gen} if and only if $\phi(x)$ is linear in $x$, we can conclude that for $\beta >3/2$ the constant $\rho(\beta)$ is sharp.
\end{rem}

\subsection{Chernoff with weight for inverse Gamma densities}

\begin{thm} [Chernoff for inverse Gamma--type densities] \label{Ch-inv}

Let $X$ be a random variable distributed with density \fer{inv-gamma} for $x\in \R_+$, $\beta > 1/2$, $m>0$. For any smooth function $\phi$   on $\R_+$ such that the variance of $\phi(X)$ finite it holds  
\be
\label{gamma-gen}
 Var[\phi(X)] \le \frac 1{\rho(\beta)} E\left\{X^2[\phi'(X)]^2\right\},
\ee
 where
\be
\rho(\beta)= 
\begin{cases}
\left(\beta-\frac 12\right)^2 & \frac 12 <\beta \leq \frac 32\\
2(\beta-1) &\beta >\frac 32.
\end{cases}
\ee
\end{thm}

\begin{proof}
The proof follows along the same lines of Theorem \ref{Ch-Cauchy}. Indeed, $h_{\beta,m}$ is the stationary state of a whole family of Fokker--Planck type equations, which depend on two positive parameters ${\alpha}$ and $\lambda$, where 
\be\label{a-l}
{2{\alpha}}+\lambda=2\beta.
\ee
For our purposes, we will take ${\alpha} \in \left( 1/2, 1\right]$ and  $\lambda >0$. Consequently, the exponent $2\beta $ of $x$ in the inverse Gamma density is greater than one.
The family of Fokker--Planck type equations having $h_{\beta,m}$ as stationary state is defined by
\be\label{FP-inv}
\partial_t h= \partial^2_x \left(x^{2{\alpha}} h\right ) +\lambda \partial_x\left(\left(x-\frac m\lambda\right ) x^{{2{\alpha}} -2} h\right), \quad x\in \R_+, t>0.
\ee
Thus, $h_{\beta,m}$ satisfies, for all $x\in\R_+$
\be
\label{sta-c}
\partial_x \left(x^{2{\alpha}} h_{\beta,m}\right ) =-\lambda \left(x-\frac m\lambda\right ) x^{{2{\alpha}} -2} h_{\beta,m}.
\ee
Equation \fer{sta-c} is of the type \fer{ste1}, with 
 \be\label{PQ}
P(x) = x^{2\alpha} ; \quad Q(x)= \lambda \left(x-\frac m\lambda\right ) x^{{2{\alpha}} -2}, \quad x\in\R_+
\ee
In the allowed range of the constant $\alpha$, the function $Q$ satisfies all the assumptions of Theorem \ref{C}.
The function $Q$ is differentiable on $\R_+$ and for  $1/2<{\alpha}\leq1$
\[
Q'(x)=\lambda({2{\alpha}}-1)x^{{2{\alpha}}-2}-m({2{\alpha}}-2)x^{{2{\alpha}}-3} \geq \lambda({2{\alpha}}-1)x^{{2{\alpha}}-2}>0, \quad  x\in \R_+.
\]
Moreover, for ${\alpha}< 1$
\[
\lim_{x\to 0^+} Q(x)=-\infty, \quad \lim_{x\to +\infty} Q(x)= +\infty.
\]
When  ${\alpha}=1$ the function $Q(x)$ is defined also for $x=0$ and $Q(0)=-m<0$. 
So $Q : \R_+ \longrightarrow \R$ is a strictly monotone, smooth transformation.
Moreover
\be\label{bound-ig}
\frac{P(x)}{Q'(x)} \leq \frac{x^{2{\alpha}}}{\lambda({2{\alpha}}-1) x^{{2{\alpha}}-2}}= \frac {x^2}{\lambda({2{\alpha}}-1) },\quad x\in\R_+.
\ee
We apply Theorem \ref {C} and  for all ${\alpha} \in \left( 1/2,1\right]$ and $\lambda >0$ satisfying \fer{a-l} we obtain
\[
 \int_{\R_+} \phi^2(x) h_{\beta,m}(x) \, d x -\left(  \int_{\R_+} \phi(x) h_{\beta,m}(x) \, d x\right )^2 \leq \frac 1{\lambda ({2{\alpha}} -1)}  \int_{\R^*} x^2  \left(\phi'(x)\right )^2 h_{\beta,m}(x) \, d x.
\]
Thanks to \fer{a-l}, we can substitute the value $\lambda = 2\beta- {2{\alpha}} $ in the constant  ${\lambda({2{\alpha}}-1) } $. This leads to maximize the constant 
\[
\rho_\beta({\alpha})= (2\beta-{2{\alpha}})({2{\alpha}}-1),
\]
with respect to $\alpha$. To this end, since
\[
\rho'_\beta({\alpha})= 2\left(2\beta-{4{\alpha}}+1\right ),
\]
we obtain
\[
\rho_\beta'({\alpha}) \geq 0 \iff {\alpha} \leq \frac {2\beta+1}4.
\]
So, if $\frac 12 < \frac {2\beta +1}4 \leq 1$, then ${\alpha}_{\max}=  \frac {2\beta +1}4$,  and if $ \frac {2\beta +1}4>1$, then ${\alpha}_{\max}=1$.
Denoting by $\rho(\beta)= \max\{\rho_\beta({\alpha}), \frac12 <{\alpha} \leq 1\}$, we obtain 
\[
\rho(\beta)= 
\begin{cases}
\left(\beta-\frac 12\right)^2 & \frac 12 <\beta \leq \frac 32\\
2(\beta-1) &\beta >\frac 32
\end{cases}
\]
and this completes the proof.
\end{proof}

\begin{rem} In the case $\beta > 3/2$ the optimal constant $\rho(\beta)$  is obtained by choosing $\alpha_{\max} =1$. In this case $Q(x) = \lambda x-m$ is therefore linear, and the proof of Chernoff inequality was already obtained in \cite{FPTT}. Moreover, in this range of the parameter $\beta$, the Cauchy-type density has finite variance and this implies that the function $\phi$ in Theorem \ref{Ch-inv} can be chosen to be linear in $x$. Since we proved in \cite{FPTT} that Chernoff inequality with weight further guarantees that there is equality in \fer{gamma-gen} if and only if $\phi(x)$ is linear in $x$, we can conclude that for $\beta >3/2$ the constant $\rho(\beta)$ is sharp.
\end{rem}

We can now rewrite the Chernoff inequality \fer{gamma-gen} in terms of the standard notation of the inverse Gamma functions with parameters  $\kappa>0$ and $m>0$, that is
 \be\label{IG}
 f_{\kappa,m} = \frac{m^\kappa}{\Gamma(\kappa)}\frac 1{x^{1+\kappa}} e^{-m/x}, \quad x \in \R_+.
 \ee
We then obtain
 \be\label{cher-gam}
 \int_{\R_+} \phi^2(x) f_{\kappa,m}(x) \, d x -\left(  \int_{\R_+} \phi(x)  f_{\kappa,m}(x) \, d x\right )^2 \leq \frac 1{\gamma(\kappa)}  \int_{\R_+} x^2  \left(\phi'(x)\right )^2  f_{\kappa,m}(x) \, d x.
\ee
In \fer{cher-gam} the optimal constants $\gamma(\kappa)$ are given by
\[
\gamma(\kappa)= 
\begin{cases}
{\kappa^2}/4 & 0< \kappa \le 2\\
\kappa -1  &\kappa >2.
\end{cases}
\]
It is immediate to check that inequality \fer{cher-gam}, for $\kappa >2$, reduces to equality when $\phi(x)$ is linear in $x$.

\section{Logarithmic Sobolev inequalities for heavy tailed densities}\label{sec:sobolev}

In this section we will apply the relationship between Cauchy-type densities and Fokker--Planck equations to obtain weighted logarithmic Sobolev inequalities in the form \fer{PoiS}. Similarly to the analysis of Section \ref{sec:chernoff}, we will refer to a suitable class of Fokker--Planck type equations \fer{FP-gen},  well adapted to the derivation of the result. Let 
\[
f_\beta(x)= \frac {C_\beta}{(1+|x|^2)^\beta}, \quad x\in \R^n
\]
denote a Cauchy--type probability density in $\R^n$, $n\geq 1$, where  
 $\beta >n/2$. 
It is known, after Bobkov and Ledoux \cite{BL}, that these densities satisfy a weighted Log-Sobolev inequality in the range  $\beta \geq \frac {n+1}2$ if $n>1$, and $\beta >1$ if $n=1$. For any probability density $f\in L_1(\mathbb{R}^n) $, absolutely continuous with respect to $f_\beta$, the inequality in \cite{BL},   can be written in the physically relevant form
\be\label{BL}
\int_{\mathbb{R}^n} f\log \frac {f }{f_\beta} dx\leq
\frac{1}{\beta -1}\int_{\mathbb{R}^n}(1+|x|^2)^2 \left|\nabla \sqrt {\frac f{f_\beta}}\right|^2\,  f_\beta\, dx.
\ee
In fact, by letting 
\[
\phi(x)= \sqrt{\frac{f(x)}{f_\beta(x)}}, \quad x\in\R^n,
\]
that implies 
 \[
 \int_{\R^n} \phi(x)^2f_\beta(x) \, dx = 1,
 \]
one easily recognizes that \eqref {BL} is equivalent to \eqref{PoiS} with weight $w(x) = (1+x^2)^2$.
Moreover, since 
\[
4  \left|\nabla \sqrt {\frac f{f_\beta}}\right|^2\,  f_\beta =  \left|\nabla\log{\frac f{f_\beta}}\right|^2\,  f
\]
inequality \eqref{BL} can be rewritten in the equivalent form
\be\label{equiv}
\int_{\mathbb{R}^n} f\log \frac {f }{f_\beta} dx\leq
\frac{1}{4(\beta -1)}\int_{\mathbb{R}^n}(1+|x|^2)^2  \left|\nabla\log{\frac f{f_\beta}}\right|^2\,  f\, dx.
\ee
In Bobkov and Ledoux \cite{BL} the weight function in inequality \fer{BL} does not depend on the value of the parameter $\beta$ characterizing the Cauchy--type density. In what follows, we are going to show that, in dimension $n=1$ the weight function in inequality \fer{BL} can be improved. 

\subsection{A sharp logarithmic Sobolev inequalities for Cauchy--type densities}

The main result of this Section is the following.

\begin{thm}[Log--Sobolev for Cauchy--type densities] \label{LS-cauchy}
Let $X$ be a random variable distributed with  the Cauchy-type probability density \fer{cau}, with $\beta >1$.  For any { bounded} smooth function $\phi(x)$, with $x\in\R$, such that  $\phi(X)$ has finite entropy, and for all  $1<\alpha< \beta$  one has the bound
\begin{equation}\label{Log_peso}
 Ent\left[\phi^2(X)\right] \le \frac{2}{\rho_{\beta, \alpha} }E\left\{(1+X^2)^\alpha [\phi'(X)]^2\right\}.
\end{equation}
 In \fer{Log_peso} the constant $\rho_{\beta, \alpha}$ is given by
\be\label{log-c}
\rho_{\beta, \alpha}= \begin{cases}
\left(2\beta-\alpha\right ) \left( \frac {\alpha-1}{2-\alpha}\right )^{3-2\alpha} & 1<\alpha <\frac 32\\
2\beta -\alpha &\frac 32 \leq \alpha<\beta, \quad \beta > \frac 32.
\end{cases}
\ee
 \end{thm}

\begin{rem} Before entering into the technical details of the proof, let us compare inequality \fer{Log_peso} with the analogous one proven  by Bobkov and Ledoux, as given by \fer{BL}. First of all, since the exponent $\alpha>1 $ of the weight function is only subject to the constraint to be less than $\beta>1$, for any value of $\beta$ we can always choose $\alpha < 2$ to satisfy the inequality. Hence we have a smaller weight, which, however, for values of $\alpha$ close to one has a worse constant $\rho_{\beta, \alpha}$. In any case, the weight $w(x)=1+x^2$ can not be reached, since $\rho_{\beta, \alpha}\to 0$ as $\alpha \to 1$. The best result  is obtained in the interval  $3/2 \le \alpha \leq 2$, since the constant  $\rho_{\beta,\alpha}> 2(\beta -1) $ and at the same time the weight function in \fer{Log_peso} is smaller than the one in \fer{BL}. 
Last, when $\beta >2$, by setting $\alpha =2$ we recover exactly  the result by Bobkov and Ledoux. \end{rem}

\begin{proof}
We proceed by proving an equivalent inequality of type \fer{equiv}, { for a smooth probability density   $f$, absolutely continuous with respect to  $f_\beta$. Then, for any bounded, smooth function $\phi$, we will consider
  $f(x)= \frac {f_\beta(x) \phi^2(x)} {\int_\R f_\beta(x) \phi^2(x)\, dx}$ and we will recover inequality \eqref{Log_peso} in the general form.}
  As in the proof of Chernoff inequality  in Theorem \ref{Ch-Cauchy}, we observe that $f_\beta$ is a stationary state of  the family of Fokker--Planck type equations
\be\label{Eqf}
\partial_t f= \partial^2_x \left((1+x^2)^\alpha f\right ) +\lambda \partial_x\left( {2\alpha x}(1+x^2)^{\alpha-1} f\right), \quad x\in \R,\, t>0.
\ee
Unlike the proof of Theorem \ref{Ch-Cauchy}, we assume now the conditions $\alpha >1$ and $\lambda >0$, still subject to the constraint $\alpha(1+\lambda)=\beta$.  This choice is coherent with the lower bound $\beta >1$ in the statement of the theorem.
In order to proceed, we make use of an equivalent formulation of the Fokker Planck equation in terms of the function
$F= \frac { f }{f_{\beta}}$. Skipping details, that can be found in \cite{FPTT},  one shows
 that $F$ satisfies the evolution equation
\begin{equation}\label{EqF}
 \frac{\partial F}{\partial t}
 =
 (1+x^2)^{\alpha} \frac{\partial^2  F}{\partial x^2} - 2\alpha \lambda x (1+x^2)^{\alpha-1} \frac{\partial F}{\partial x}.
\end{equation}

Following the original argument of Feller \cite{Fe52}, we introduce a change of variables to make the diffusion coefficient equal to unity.
To this end, let us define
\be\label{FG}
G(y,t)=F(x,t), 
\ee
with
\begin{equation}\label{cambio}
 \frac{dy}{dx}=\frac{1}{(1+x^2)^{{\alpha}/{2}}}.
 \end{equation}
Owing to \fer{cambio} we obtain 
$$
\frac{\partial F}{\partial x}=  \frac{1}{(1+x^2)^{{\alpha}/{2}}} \frac{\partial G}{\partial y}
$$
and
$$
\frac{\partial^2 F}{\partial x^2}= \frac{1}{(1+x^2)^{{\alpha}}} \frac{\partial^2 G}{\partial y^2} - \frac{\alpha x}{(1+x^2)^{{\alpha}/2 +1}} \frac{\partial G}{\partial y}.
$$
Therefore the right hand side of \eqref{EqF} becomes
$$
\frac{\partial^2 G}{\partial y^2}
- 
\alpha x(1+x^2)^{\frac \alpha 2-1} \frac{\partial G}{\partial y}
 -
2\alpha \lambda x (1+x^2)^{\frac \alpha 2-1}
 \frac{\partial G}{\partial y}.
$$
We denote by $x=x(y)$  the inverse  of the increasing function $y(x),$ defined by (\ref {cambio}).
Hence equation (\ref{EqF}) turns into a Fokker Planck equation with coefficient of diffusion equal to one
\begin{equation}\label{EqG}
 \frac{\partial G}{\partial t}=\frac{\partial^2 G}{\partial y^2}
- W'(y) \frac{\partial G}{\partial y},
 \end{equation}
where $W'(y)$ is the drift term 
  \begin{equation}\label{Wprimo}
  W^{\prime}(y)=
\alpha (1+2\lambda)x(y)
 (1+x^2(y))^{\frac \alpha 2 -1}.
\end{equation}
Equation (\ref{EqG}) is the adjoint of the Fokker--Planck equation
\begin{equation}\label{Eqg}
\frac{\partial g}{\partial t}=
\frac{\partial^2 g}{\partial y^2} +
\frac{\partial }{\partial y}(  W^{\prime}(y)
g)
\end{equation}
still with diffusion coefficient  equal to one, and steady state 
 \be\label{www}
g_{\infty}(y)= C e^{-W(y)}.
\ee
As shown in \cite{To20},  it is useful to introduce a further version of the Fokker--Planck equation \fer{EqF}, that highlights an interesting feature of the change of variables \fer{cambio}. 
 For given $t >0$, let $X(t)$ denote the random process with probability density $f(x,t)$, solution of the Fokker--Planck equation \fer{Eqf}, and let
 \be\label{dist}
 \MF(x,t) = P(X(t) \le x)=  \int_{-\infty}^x f(y,t)\, dy
 \ee
denote its probability distribution. 
Integrating both sides of equation \fer{Eqf} on  $(-\infty,x)$,  it follows by simple computations that $\MF(x,t)$ satisfies the equation 
\begin{equation}\label{EqF1}
 \frac{\partial \MF}{\partial t}
 =
 (1+x^2)^{\alpha} \frac{\partial^2  \MF}{\partial x^2} +2\alpha(1+ \lambda) x (1+x^2)^{\alpha-1} \frac{\partial \MF}{\partial x}.
\end{equation}
As before, let us define
\be\label{iden}
\MG(y,t)=\MF(x,t), 
\ee
where $y=y(x)$ is defined through \fer{cambio}. Then, using the same computations leading from \fer{EqF} to \fer{EqG} it is immediate to show that $\MG$ satisfies 
\begin{equation}\label{EqG1}
 \frac{\partial \MG}{\partial t}=\frac{\partial^2 \MG}{\partial y^2}
+W'(y) \frac{\partial \MG}{\partial y}.
 \end{equation}
Hence, if for given $t >0$,  $Y(t)$ denotes the random process with probability density $g(x,t)$, solution of the Fokker--Planck equation \fer{Eqg},  $\MG(y,t)$ is the distribution function of the process $Y(t)$. This relation implies an explicit connection between the solutions to the equations \fer{Eqf} and \fer{Eqg}. Indeed, differentiating the identity \fer{iden}, one obtains for all $t \ge 0$
 \be\label{ide}
g(y(x),t)= f(x, t) (1+x^2)^{\frac{\alpha}{2}},
 \ee
and
\be\label{iss}
g_{\infty}(y(x))= f_{\beta}(x) (1+x^2)^{\frac{\alpha}{2}}. 
\ee
The properties of the steady state $g_\infty(y)$ can be  easily deduced from \fer{iss}.  Recalling that $\alpha >1$,
 the change of variable \eqref{cambio} implies 
\[
y(x)=\int_0^x \frac{1}{(1+t^2)^{\frac\alpha 2}}\, dt, 
\]
and since the integral function belongs to $L_1(\R)$, then  $\lim_{x\to \pm \infty}y(x)= \pm a(\alpha) $. Thus, $y(x)$ is contained in the strip  $\mathbb{M}=[-a,a]$.
We are now  ready to prove  inequality (\ref{Log_peso}). Actually, Fokker--Planck equations  of type (\ref{Eqg})  have been introduced as a useful working tool to get logarithmic Sobolev inequalities for probability densities  different from the standard Gaussian \cite{OV}. The argument follows from Bakry and Emery theorem \cite{BE}, which can be immediately applied thanks to the particular form of (\ref{Eqg}).
More precisely, given the equilibrium density $g_{\infty}=C e^{-W(y)}$ defined on a complete manifold  $\mathbb{M}=[-a,a]\subset \mathbb{R}$,  Bakry and Emery criterion guarantees that for all smooth probability densities $g$ on $\mathbb M$ absolutely continuous  with respect $g_{\infty}$,
it holds
\begin{equation}\label{Bak}
\int_{\mathbb{M}} g(y)\log \frac {g(y) }{g_{\infty}(y)} dy\leq
\frac{1}{2\rho}\int_{\mathbb{M}}\left(\frac{d}{dy}
\log\frac {g(y)}{g_\infty(y)}\right)^2g(y) dy,
\end{equation}
provided that the function $W$ is strongly convex, with
\begin{equation}\label{Wsecondo}
W^{\prime\prime}(y) \geq \rho>0.
\end{equation}
In our case
\[
W(y)= \int_0^y \alpha (1+2\lambda)x(s)
 (1+x^2(s))^{\frac \alpha 2 -1}\, ds.
 \]
Resorting to condition (\ref{cambio}) we easily obtain
\[
W^{\prime\prime}(y) =
\alpha (1+2\lambda) 
\frac{1+(\alpha-1) x^2(y)}{\left( 1+x^2(y) \right)^{2-\alpha}}
\]
The even function 
 \[
z(x)=\frac{1+(\alpha-1) x^2(y)}{\left( 1+x^2(y) \right)^{2-\alpha}}.
 \]
attains its minimum in the point $\bar x=0$ if $\alpha \geq \frac{3}{2}$, and in the point $\bar x=
\frac{(3-2\alpha)^{\frac{1}{2}}}{\alpha-1},$ if $1<\alpha < \frac{3}{2}$. Consequently
\[
W^{\prime\prime}(y) \geq \alpha(1+2\lambda),
\quad \alpha > \frac{3}{2}
\]
whereas
\[
W^{\prime\prime}(y) \geq  \alpha(1+2\lambda) z(\bar x)=  \alpha(1+2\lambda)\left(
\frac{\alpha-1}{2-\alpha}\right )^{3-2\alpha}, \quad 1<\alpha \leq \frac32.
\]
Let us notice that, as $\alpha \to 1$,  the convexity condition is lost.

Finally, for $\alpha >1$, and for any smooth probability density function $g$ absolutely continuous with respect to $g_\infty$, we get the logarithmic Sobolev inequality
\begin{equation}\label{Bak1}
\int_{\mathbb{M}} g(y)\log \frac {g(y)}{g_{\infty}(y)} dy\leq
\frac{1}{2\rho}\int_{\mathbb{M}}\left(\frac{d}{dy}
\log\frac {g(y)}{g_{\infty}(y)}\right)^2 g(y) dy
\end{equation}
with
$\rho= \alpha(1+2\lambda)z(\bar x)=\rho_{\alpha,\lambda}.$

The last step relies in rewriting inequality \fer{Bak1} in terms of the original Cauchy density $f_{\beta}$. 
 This can be obtained easily by resorting again  to the change of variables \eqref{cambio}. In view of \fer{FG} and  \fer{ide}   the integral on the left-hand side of
(\ref{Bak1}) becomes
\[
\int_{-\infty}^{+\infty}g(y(x))\left( \log \frac{g(y(x))}{g_{\infty}(y(x))}\right)
\frac{1}{(1+x^2)^{\frac{\alpha}{2}}}\,
dx = \int_{\mathbb{R}} f(x) \log \frac{f(x)}{f_{\beta}(x)} \, dx.
\]
Likewise, the integral on the right-hand side of \fer{Bak1}  becomes
\begin{equations}\nonumber
&\int_{-\infty}^{+\infty}
\left(
\frac{d}{dx}
\log\frac{g(y(x))}{g_{\infty}(y(x))}\right)^2 
(1+x^2)^{\alpha}g(y(x))
\frac{1}{(1+x^2)^{\frac{\alpha}{2}}}dx \\ & \int_{\mathbb{R}}(1+x^2)^{\alpha}\left(\frac{d}{dx}
\log\frac{ f(x)}{f_\beta(x)}\right)^2 \, f(x) \, dx.
\end{equations}
Finally, inequality \fer{Bak1}, written in terms of $f$ and $f_\beta$ reads
\[
\int_{\mathbb{R}} f(x)\log \frac{ f(x)}{f_\beta(x)} \, dx\leq
\frac{1}{2\rho_{\alpha, \lambda} }\int_{\mathbb{R}}(1+x^2)^{\alpha}\left(\frac{d}{dx}
\log\frac{ f(x)}{f_\beta(x)}\right)^2 \, f(x)\, dx.
\]
Resorting to the relation
\[
\beta= \alpha(1+\lambda)
\]
we replace $\lambda= \frac \beta \alpha -1$ and we get $\rho_{\alpha, \lambda}= \rho_{\beta, \alpha}$
with
\[
\rho_{\beta, \alpha}= \begin{cases}
\left(2\beta-\alpha\right ) \left( \frac {\alpha-1}{2-\alpha}\right )^{3-2\alpha} & 1<\alpha <\frac 32\\
2\beta -\alpha &\frac 32 \leq \alpha<\beta.
\end{cases}
\]
This concludes the proof.
\end{proof}
\subsection{Weighted logarithmic Sobolev inequalities for inverse Gamma densities}

As discussed in the Introduction, sharp  logarithmic Sobolev inequalities for inverse--Gamma type densities are  directly connected to the study of convergence to equilibrium for Fokker--Planck type equations like \fer{FPz}, of interest in the study of wealth distribution in a western society.  The  result that follows is contained in the paper \cite{FPTT20} and it is here reported with few details to make it possible to compare it with the result for the Cauchy--type densities obtained in the previous Section. Like in Section \ref{sec:chernoff}, we use expression \fer{inv-gamma} that allows for a direct comparison with the result of Theorem \ref{LS-cauchy}.

\begin{thm} [Log--Sobolev for inverse Gamma--type densities] \label{LS-inv}

Let $X$ be a random variable distributed with  the inverse Gamma probability density $h_{\beta,m}(x)$ defined by \fer{inv-gamma}, with $\beta >1$, $m>0$.  For any { bounded}  smooth function $\phi(x)$, with $x\in\R_+$ such that  $\phi(X)$ has finite entropy, and for all  for all $1<\alpha\leq\frac32$ and $\alpha<\beta$   one has the bound
\begin{equation}\label{Peso-g}
 Ent\left[\phi^2(X)\right] \le \frac{2}{\rho_{\beta, \alpha,m} }E\left\{X^{2\alpha} [\phi'(X)]^2\right\}.
\end{equation}
In inequality  \fer{Peso-g}  $ \rho_{\beta, \alpha, m}$ is given by
\be
 \rho_{\beta, \alpha, m}= \begin{cases} 
 \frac 12 \left( \frac{2\beta-\alpha}{\frac 32 -\alpha}\right )^{3-2\alpha} (m(2-\alpha))^{2\alpha-2} (\alpha-1)^{5-4\alpha} & 1<\alpha <\frac 32\\
\frac m2 & \alpha=\frac 32, \quad \beta > \frac 32.
\end{cases}
\ee
 \end{thm}

  \begin{proof}
We proceed as in Theorem \ref{LS-cauchy}, by proving the equivalent inequality
 \be
\label{Peso-g2}
\int_{\mathbb{R_+}} h(x) \log \frac{h(x)}{h_{\beta,m}(x)} \, dx\leq
\frac{1}{2 \rho_{\beta, \alpha, m}}\int_{\mathbb{R_+}}x^{2\alpha}\left(\frac{d}{dx}
\log\frac{ h(x)} {h_{\beta,m}(x)}\right)^2 \, h(x) \, dx
\ee
{ for any probability density $h$, smooth and absolutely continuous with respect to $h_{\beta,m}$.}
 As shown in Theorem \ref{Ch-inv}, $h_{\beta,m}$ is the stationary state of the family of Fokker--Planck type equations \fer{FP-inv},   depending on the  two positive parameters ${\alpha}$ and $\lambda$ satisfying the constraint \fer{a-l}. Unlike Theorem \ref{Ch-inv},  we assume now ${\alpha} \in \left(1,\frac 32\right]$ and $\lambda >0$ so that we can treat all  $\beta >1$.
In terms of the function
$H= \frac { h}{h_{\beta,m}}$, \fer{FP-inv} reads
\be\label{eqH}
\partial_t H=  x^{2{\alpha}}\partial^2_x H - \lambda \left(x-\frac m\lambda\right ) x^{{2{\alpha}} -2} \partial_x H, \quad x\in \R_+, t>0.
\ee
As in Theorem \ref{LS-cauchy}  we  change variable  to transform the Fokker--Planck type equation \fer{eqH} into a new one with  coefficient of diffusion equal to one. This is done by setting
$$
L(y,t)=H(x,t), 
$$
with
\begin{equation}\label{cambio2}
 \frac{dy}{dx}=-\frac 1{x^\alpha}, \quad x\in \R_+.
 \end{equation}
In terms of $L$, the right-hand side of \eqref{eqH} becomes
$$
\frac{\partial^2 L}{\partial y^2}
- \left(mx^{\alpha-2} -(\alpha+\lambda) x^{\alpha-1}\right ) \frac{\partial L}{\partial y}
$$
where $x=x(y)$ is the inverse  of the decreasing function $y(x),$ defined by (\ref {cambio2}).
In this case the function $y(x)$ can be computed explicitly to give
\be\label{cambio-var}
y(x)=\frac 1{\alpha-1} \frac 1{x^{\alpha-1}}
\ee
so that $y\in \R_+$. Equation \eqref{eqH} turns   into
\be\label{eqG2}
\frac {\partial L}{\partial t} = \frac{\partial^2 L}{\partial y^2}
- U'(y) \frac{\partial L}{\partial y},
\ee
where the drift term equals
  \begin{equation}\label{Wprimo2}
  U^{\prime}(y)=
m (\alpha-1)^{\frac{2-\alpha}{\alpha-1}}
y^{ \frac{2-\alpha}{\alpha-1}} - 
\frac {\alpha+\lambda}{\alpha-1} \frac 1y.
\end{equation}
Equation (\ref{eqG2}) is the adjoint of the Fokker--Planck equation
\begin{equation}\label{Eqg2}
\frac{\partial l}{\partial t}=
\frac{\partial^2 l}{\partial y^2} +
\frac{\partial }{\partial y}(  U^{\prime}(y)
l)
\end{equation}
with diffusion coefficient still equal to one and steady state $l_{\infty}(y)= C e^{-U(y)}$.
In this case, we recognize that $l_\infty$ is a generalized Gamma density \cite{Sta} 
\[
l_{\beta, \alpha, m}(y)= \frac {C_{\beta, \alpha, m}}{y^{\frac {2\beta -\alpha}{\alpha-1}}} e^{-m(\alpha-1)^{\frac 1{\alpha-1}}y^{\frac 1{\alpha-1}}}, \quad y\in \R_+.
\]
Proceeding as in the proof of Theorem \ref{LS-cauchy}, we conclude that the relation between the inverse--Gamma density $h_{\beta,m}$ and the generalized Gamma density $l_{\beta, \alpha, m}$ is given by
\be\label{relazione}
h_{\beta,m}(x) = l_{\beta, \alpha, m}(y(x))\left| \frac{dy}{dx}\right|.
\ee
To apply  Bakry and Emery criterion to $l_{\beta, \alpha, m}$, we find a positive lower bound on $U^{\prime\prime}$. Since
\be\label{wsec}
U^{\prime\prime}(y)= \frac 1{y^2(\alpha-1)}  \left(m (2-\alpha) (\alpha-1)^{\frac {2-\alpha}{\alpha-1}} y^{\frac {1}{\alpha-1}}+{\alpha+\lambda}\right ), \quad y>0,
\ee
for $\alpha =\frac 32$ we have
\be\label{ro-tremezzi}
U^{\prime\prime}(y) \geq \frac m2 := \rho\left( \beta, \frac 32, m \right), \quad y>0.
\ee
If $1<\alpha<\frac 32$, then $U^{\prime\prime}$ achieves its minimum in 
\[
\bar y= \left( \frac {\alpha+\lambda}{m(2-\alpha)\left(\frac 32-\alpha\right )}\right )^{\alpha-1} \frac 1{(\alpha-1)^{3-2\alpha}}.
\]
Owing to \fer{a-l} we write
\[
\lambda= 2\beta -2\alpha.
\]
Then  the minimum of the function $U^{\prime\prime}$ is given by
\[
U^{\prime\prime}(\bar y):= \rho_{\beta,\alpha, m} >0
\]
with 
\be\label{ro-gen}
 \rho_{\beta, \alpha, m}=
 \frac 12 \left( \frac{2\beta-\alpha}{\frac 32 -\alpha}\right )^{3-2\alpha} (m(2-\alpha))^{2\alpha-2} (\alpha-1)^{5-4\alpha}.
 \ee
 It is easy to verify that
\[
\lim_{\alpha \to \frac 32}  \rho_{\alpha, m, \beta} =\frac m2.
\]
We remark that if $\alpha >\frac 32$ 
\[
\lim_{y\to +\infty}  U^{\prime\prime}(y)=0, 
\]
and the strict convexity of $U(y)$ is lost.

If $\alpha \le 3/2$  we apply Bakry--Emery criterion as in  \cite{FPTT19} and we get the logarithmic Sobolev inequality for the generalized Gamma density $l_{\beta, \alpha, m}$
\be\label{LS-gammagen}
\int_{\R_+}  l(y) \log \frac {l(y)}{l_{\beta, \alpha, m}(y)} dy \leq \frac 1{2\rho} \int_{\R_+}  \left(
\frac{d}{d y} \log  \frac {l(y)}{l_{\beta, \alpha, m}(y)} 
\right )^2 l(y) dy
\ee
where $\rho= \rho_{\beta, \alpha, m}$ as in \eqref{ro-tremezzi} and \eqref{ro-gen}.
Turning back to the original variables gives the result.
\end{proof}

As in Section \ref{sec:chernoff}, we can rewrite inequality \fer{Peso-g} in terms of the standard notation of the inverse Gamma functions \fer{IG}. We then obtain
\be
\label{Peso-gg}
\int_{\mathbb{R_+}} h(x) \log \frac{h(x)}{h_{\kappa,m}(x)} \, dx\leq
\frac{1}{2 \rho_{\kappa, \alpha, m}}\int_{\mathbb{R_+}}x^{2\alpha}\left(\frac{d}{dx}
\log\frac{ h(x)} {h_{\kappa,m}(x)}\right)^2 \, h(x) \, dx,
\ee
or, equivalently, if $X$ is a random variable distributed with probability density function \fer{IG}
\begin{equation}\label{Peso-g3}
 Ent\left[\phi^2(X)\right] \le \frac{2}{\rho_{\kappa, \alpha,m} }E\left\{X^{2\alpha} [\phi'(X)]^2\right\}.
\end{equation}
In inequalities  \fer{Peso-gg} and \fer{Peso-g3} the constant $ \rho_{\kappa, \alpha, m}$ is given by
\be
 \rho_{\kappa, \alpha, m}= \begin{cases} 
 \frac 12 \left( \frac{\kappa +1-\alpha}{\frac 32 -\alpha}\right )^{3-2\alpha} (m(2-\alpha))^{2\alpha-2} (\alpha-1)^{5-4\alpha} & 1<\alpha <\frac 32\\
\frac m2 & \alpha=\frac 32, \quad \kappa > 2.
\end{cases}
\ee

\section{Wirtinger-type inequalities for heavy tailed densities}\label{sec:wirtinger} 

Let $X$ be a random variable with an absolutely continuous density $f(x)$, $x \in \MI = (i_-,i_+) \subseteq \R$
such that $f(x) >0$ in  $\MI$, and let $F(x)$, $x \in \MI$, denote its distribution function, defined as usual by the formula
 \be\label{distri}
 F(x) = \int_{i_-}^x f(y)\, dy \le 1.
 \ee
Let $\bar x$ denote the median of the random variable $X$, that is the value where the increasing function  $F(x)$ satisfies  $F(\bar x)= 1/2$. Last, let $K(x)$ be defined as the nonnegative function 
\be\label{peso-W}
K(x) = \frac{F(x)}{f(x)} \quad {\rm{if}}\,\, x \le \bar x; \quad K(x) = \frac{1- F(x)}{f(x)} \quad {\rm{if}}\,\,  x \ge \bar x.
\ee
Then, $K(x)$ is a continuous function on $\MI$, and we have the identity
 \be\label{chiave}
 f(x) = - \frac{x-\bar x}{|x- \bar x|}\frac d{dx}\left[K(x) f(x) \right].
 \ee
Note that \fer{chiave} is a clean way to characterize the density $f(x)$ as the steady state of a Fokker--Planck equation of type \fer{FP-gen} where the diffusion coefficient is the continuous nonnegative function
 \[
 P(x) = K(x),
 \]
and the drift term is
 \[
 Q(x) = \frac{x-\bar x}{|x- \bar x|}. 
 \]
Note moreover that the drift term defined above satisfies  conditions \fer{bordi} at the boundaries of $\MI$.
Using expression \fer{chiave} we prove the following

\begin{thm}[Wirtinger with weight]\label{W}
Let $X$ be a random variable distributed with density $f(x)$, $x \in \MI =(i_-, i_+)\subseteq \R$, and let $K(x)$ be defined by \fer{peso-W}. Then,  for any smooth function $\phi$    on $\MI$ such that $ E\left[ |\phi(X)|^p \right]$ is bounded,  $1 \le p<+\infty$, it holds 
 \be\label{Wi-gen}
 E\left[ |\phi(X)- E(\phi(X))|^p \right] \le (2p)^{\null \,p} E\left[ K(X)^{p}\, |\phi'(X)|^p \right]. 
 \ee
\end{thm}

\begin{proof}
Let us first suppose that the function $\phi$ satisfies the condition $\phi(\bar x) =0$. In this case, we can directly make use of the argument of proof in \cite{EM}. Thanks to \fer{chiave}, we have
 \begin{equations}\label{pp}
& \int_{i_-}^{\bar x} |\phi(x)|^p f(x) \, dx =  \int_{i_-}^{\bar x} |\phi(x)|^p \frac d{dx}\left[K(x) f(x) \right]\, dx = \\
& |\phi(x)|^p K(x) f(x) \Big|_{i_-}^{\bar x} -  \int_{i_-}^{\bar x}  K(x) f(x) \frac d{dx} |\phi(x)|^p\, dx.
 \end{equations}
 Now, since $\phi(\bar x) =0$,
 \[
 |\phi(x)|^p K(x) f(x) \Big|_{i_-}^{\bar x} = |\phi(\bar x)|^p F(\bar x) -  \lim_{x \to i_-} |\phi(x)|^p F( x) = -  \lim_{x \to i_-} |\phi(x)|^p F( x) \le 0,
 \]
 and the contribution of the boundary term  is nonpositive on the interval $(i_-, \bar x)$. Therefore \fer{pp} implies  the inequality
  \[
   \int_{i_-}^{\bar x}  |\phi(x)|^p f(x) \, dx \le p \int_{i_-}^{\bar x}  K(x) f(x) |\phi(x)|^{p-1} |\phi'(x)|\, dx.
  \]
The same argument can be used on the interval $(\bar x, i_+)$, to obtain
 \[
   \int_{\bar x}^{i_+} |\phi(x)|^p f(x) \, dx \le p   \int_{\bar x}^{i_+} K(x) f(x) |\phi(x)|^{p-1} |\phi'(x)|\, dx.
  \]
Consequently, if $\phi(\bar x) =0$, we have the inequality
 \be\label{phi-0}
   \int_\MI |\phi(x)|^p f(x) \, dx \le p  \int_\MI K(x) f(x) |\phi(x)|^{p-1} |\phi'(x)|\, dx.
  \ee
If $p =1$ , \fer{phi-0} reduces to 
  \be\label{ine1}
E\left[ |\phi(X)| \right] \le  E\left[K(X)\, |\phi'(X)| \right].    
  \ee
If $1< p <+\infty$, H\"older's inequality shows that
 \begin{equations}
& \int_\MI K(x) f(x) |\phi(x)|^{p-1} |\phi'(x)|\, dx \le \\
&\left[\int_\MI\left( K(x) |\phi'(x)|\right)^p f(x)\, dx \right]^{1/p} \left[\int_\MI |\phi(x)^p f(x)\, dx \right]^{1- 1/p},
 \end{equations}
 which, combined with \fer{phi-0}, shows that, for any function $\phi$ satisfying $\phi(\bar x) = 0$, it holds
  \be\label{Wi-best}
 E\left[ |\phi(X)|^p \right] \le p^{\null \,p} E\left[ K(X)^{p}\, |\phi'(X)|^p \right]. 
 \ee
 The general case is an easy consequence of the previous argument. Indeed, since $f(\cdot)$ is a probability density on $\MI$, for $1 \le p < +\infty$ we have 
  \begin{equations} \label{gene}
  & \int_\MI \left| \phi(x) - \int_\MI \phi(y)\,f(y)\,dy \right|^p f(x) \,dx = \int_\MI \left| \int_\MI \left(\phi(x) -  \phi(y)\right)\,f(y)\,dy \right|^p f(x) \,dx  \le  \\
  & \int_{\MI \times\MI} \left| \phi(x) -  \phi(y)\right|^p f(x) f(y) \,dx \,dy = \\
  & \int_{\MI \times\MI} \left| \phi(x) -\phi(\bar x) - ( \phi(y)-\phi(\bar x) )\right|^p f(x) f(y) \,dx \,dy \le \\
  &2^{p-1} \int_{\MI \times\MI} \left(\left| \phi(x) -\phi(\bar x)\right|^p +\left| \phi(y)-\phi(\bar x)\right|^p\right) f(y) f(x) dx \,dy = \\
  & 2^p  \int_{\MI } \left| \phi(x) -\phi(\bar x) \right|^p f(x) \,dx  = 2^p  \int_{\MI} \left| \psi(x) \right|^p f(x) \,dx, 
   \end{equations}
 where the function $\psi(x)$ in \fer{gene} is such that $\psi(\bar x) = 0$. At this point, we can apply \fer{Wi-best} to the function $\psi$ to get the general inequality \fer{Wi-gen}.
 \end{proof}

Unlike the result of \cite{EM}, the function $\phi$ is not required to satisfy particular boundary conditions at the point $i_-$. For example, it is not necessary, in the case $\MI = \R_+$, that, as required by  Corollary to Theorem {\rm 1} of \cite{EM}, the function $\phi$ satisfies $\phi(0) =0$. 

\subsection{Wirtinger inequalities with weight for Cauchy-type densities}
In this short Section, we apply Theorem \ref{W} to recover inequalities for the class of Cauchy-type densities, with an explicit expression of the weight function $K(x)$. We prove

\begin{thm}
Let $X$ be a random variable distributed with the Cauchy-type density \fer{cau}, with  $\beta >1/2$. For any smooth function $\phi(x)$, with $x\in\R$, such that  $ E\left[ |\phi(X)|^p \right]$ is bounded, $1\le p< +\infty$, one has the inequality
\be
\label{wirtinger-gen}
 E\left[ |\phi(X)- E(\phi(X))|^p \right] \le 2^\beta \left( \frac{2p}{2\beta-1}\right)^{\null \,p} E\left[ (1+|X|)^{p}\, |\phi'(X)|^p \right].
\ee
\end{thm}

\begin{proof}
Let $g_\beta(x)$, $\beta >1/2$,  denote the class of probability density functions in $\R$ given by
 \be\label{cc}
 g_\beta(x) = \frac{c_\beta}{(1+ |x|)^{2\beta}}.
 \ee
 Since the densities $g_\beta(x)$ are symmetric, the median is $\bar x =0$, and it is immediate to show that, for any given $\beta$ the weight function of $g_\beta(x)$ is given by
 \be\label{wei-1}
K(x) = \frac{1+|x|}{2\beta -1}.  
 \ee
Hence, if the random variable $Y$  is distributed with density $g_\beta(x)$, for any given $1\le p <+\infty$, Theorem \ref{W} implies the inequality
\be\label{W-1}
 E\left[ |\phi(Y)- E(\phi(Y))|^p \right] \le  \left( \frac{2p}{2\beta-1}\right)^{\null \,p} E\left[ (1+|Y|)^{p}\, |\phi'(Y)|^p \right].
 \ee
Inequality \fer{wirtinger-gen} for the Cauchy-type densities then follows from \fer{W-1} by resorting to  the chain of elementary inequalities
 \[
f_\beta(x) \le 2^\beta g_\beta(x) \le 2^\beta f_\beta (x).
 \]

\end{proof}

Note that,  if the function $\phi(x)$ is such that $\phi(0) =0$, the random variable $Y$ satisfies inequality \fer{Wi-best}, that in this case reads 
 \be\label{Wi-4}
 E\left[ |\phi(Y)|^p \right] \le  \left( \frac{p}{2\beta-1}\right)^{\null \,p} E\left[ (1+|Y|)^{p}\, |\phi'(Y)|^p\right].
  \ee
Inequality \fer{Wi-4} is sharp, since the weight function $K(x) $ is exact. This sharpness is lost for the class of Cauchy-type densities. We remark that the weight function $K(x)$, with different constants, has been obtained in \cite{BL} in the case $p=1$, in any dimension $n \ge 1$.

\subsection{Wirtinger inequalities  with weight for inverse Gamma densities}

Last,  we apply Theorem \ref{W} to recover inequalities for the class of inverse Gamma densities. In this case, the expression of the weight function $K(x)$ depends on the value of the median of the distribution, which is not explicitly available. We prove

\begin{thm}  \label{W-inv}

Let $X$ be a random variable distributed with density $h_{\beta,m}$ defined as in \fer{inv-gamma}, for $x\in \R_+$, $\beta > 1/2$, $m>0$. For any smooth function $\phi$    on $\R_+$ such that  $ E\left[ |\phi(X)|^p \right]$ is finite it holds  
\be
\label{gamma-gen-W}
E\left[ \left|\phi(X) -E(\phi(X))\right|^p\right] \le { \left(p D(\beta,m)\right )^p}  E\left\{X^p[\phi'(X)]^p\right\},
\ee
 where
\be
D(\beta, m) = \frac 1{ \bar x_{\beta,m} h_{\beta,m}\left(\bar x_{\beta,m}  \right)},
\ee
and $\bar x_{\beta,m}$ is the median of the random variable $X$.
\end{thm}

\begin{proof}
For any pair of positive constants $\beta,m$, let $\bar x_{\beta,m}$ denote the median of the random variable $X$ with density $h_{\beta,m}$ given by  \fer{inv-gamma}, and distribution function $H_{\beta,m}(x)$. Then, if $x \le    \bar x_{\beta,m}$
 \begin{equations}\label{meno}
 \frac{H_{\beta,m}(x)}{h_{\beta,m}(x)}  = &\int_0^x \left(\frac xy \right)^{2\beta} \exp\left\{ -\frac mx\left( \frac xy -1\right)\right\}\, dy = \\
 &x\, \int_0^1z^{-2\beta} \exp\left\{ -\frac mx\left( \frac 1z -1\right)\right\}\, dz \le \\
 &x\, \int_0^1z^{-2\beta} \exp\left\{ -\frac m{ \bar x_{\beta,m}}\left( \frac 1z -1\right)\right\}\, dz.
  \end{equations}
Indeed, the value of the integral on the second line of \fer{meno} is non  decreasing with respect to $x$, as it can be easily verified by direct inspection. Likewise, if $x \ge   \bar x_{\beta,m}$ one shows that
\begin{equations}\label{piu}
 \frac{1 - H_{\beta,m}(x)}{h_{\beta,m}(x)}  = &\int_x^{+\infty} \left(\frac xy \right)^{2\beta} \exp\left\{ -\frac mx\left( \frac xy -1\right)\right\}\, dy \le \\
 &x\, \int_1^{+\infty}z^{-2\beta} \exp\left\{ -\frac m{ \bar x_{\beta,m}}\left( \frac 1z -1\right)\right\}\, dz.
  \end{equations}
On the other hand we have
\begin{equations}\label{meno1}
  &\int_0^1z^{-2\beta} \exp\left\{ -\frac m{ \bar x_{\beta,m}}\left( \frac 1z -1\right)\right\}\, dy = \\
& e^{  m/{ \bar x_{\beta,m}}} \int_0^1z^{-2\beta} \exp\left\{ -\frac m{ \bar x_{\beta,m}z} \right\}\, dz =\\
&\frac 1{C_{\beta,m}} e^{  m/{ \bar x_{\beta,m}}} (\bar x_{\beta,m})^{2\beta -1} \int_0^{\bar  x_{\beta,m}}C_{\beta,m} u^{-2\beta} \exp\left\{ -\frac m{ u} \right\}\, du = \\
&\frac 1{C_{\beta,m}} e^{  m/{ \bar x_{\beta,m}}} (\bar x_{\beta,m})^{2\beta -1} \int_0^{\bar  x_{\beta,m}}h_{\beta,m}(u) \, du = \frac 12\left[ \bar x_{\beta,m} h_{\beta,m}( \bar x_{\beta,m})\right]^{-1}.
  \end{equations}
In fact, by definition of median, the last integral into \fer{meno1} is equal to $1/2$. Clearly, the same result holds for the last integral into \fer{piu}. This concludes the proof.

\end{proof}

\section{Conclusions}
The recent developments of mathematical modeling of social and economic phenomena led to the study of new types of Fokker--Planck equations characterized by steady state solutions with fat tails. For a precise study of the convergence to equilibrium of the solution to these equations, functional inequalities with weight are the main mathematical tool. In this paper we showed how to make use of these Fokker--Planck type equations to obtain one-dimensional functional inequalities in sharp form. This method is closely connected to kinetic theory, and more in general to statistical physics, and gives a new light to the meaning of these inequalities, in agreement with the results obtained in the case of the classical Fokker--Planck equation \cite{MV, To97,To99}. 

\vskip 2cm
\section*{acknowledgements}
This work has been written within the activities of GNFM (Gruppo Nazionale per la Fisica Matematica) and of GNAMPA (Gruppo Nazionale per l'Analisi Matematica, la Probabilità e le loro Applicazioni) of INdAM (Istituto Nazionale di Alta Matematica), Italy.
The research was partially supported by the Italian Ministry of Education, University and Research (MIUR) through the ``Dipartimenti di Eccellenza'' Programme (2018-2022) -- Department of Mathematics ``F. Casorati'', University of Pavia  and through the MIUR  project PRIN 2017TEXA3H ``Gradient flows, Optimal Transport and Metric Measure Structures''.  The authors states that there is no conflict of interest.

\end{document}